\documentclass[10pt, a4paper]{amsart}
\usepackage[english]{babel}
\usepackage{amssymb,amsmath,amsthm,mathrsfs}
\usepackage{verbatim}
\usepackage[margin=1in]{geometry}

\numberwithin{equation}{section}

\newtheorem{theorem}{Theorem}
\newtheorem{corollary}{Corollary}
\newtheorem{lemma}{Lemma}

\def\bal{\begin{aligned}}
\def\eal{\end{aligned}}
\def\be{\begin{equation}\label}
\def\ee{\end{equation}}
\def\bcs{\begin{cases}}
\def\ecs{\end{cases}}
\def\={\;=\;}
\def\+{\,+\,}
\def\-{\,-\,}

\def\Z{{\mathbb Z}}

\def\Q{{\mathbb Q}}
\def\R{{\mathbb R}}
\def\F{{\mathbb F}}
\def\P{{\mathbb P}}
\def\L{\Lambda}
\def\Newt{{\rm Newt}}
\def\supp{{\rm supp}}
\def\cong{\;\equiv\;}

\title{Explicit $p$-adic unit-root formulas for hypersurfaces}
\author{Masha Vlasenko}

\begin{document}
\date{\today}
\maketitle

\begin{abstract}We prove a statement on $p$-adic continuity of matrices of coefficients of the logarithm of the Artin-Mazur formal group law associated to the middle cohomology of a hypersurface. As J.~Stienstra discovered in 1986, the entries of these matrices are coefficients of powers of the equation of the hypersurface, and in certain cases they satisfy congruences of Atkin and Swinnerton-Dyer type. These congruences imply that eigenvalues of our limiting matrices are eigenvalues of Frobenius of zero $p$-adic valuation on the middle crystalline cohomology of the fibre at $p$.\end{abstract}

\tableofcontents

This note was written in preparation for the talk given at Max Plank Institute for Mathematics in Bonn on the 14th of January 2015. The author is grateful to Jan Stienstra for helpful conversations on the subject a while ago, in Bonn and Utrecht. Our main result is on the $p$-adic continuity, which was suggested by him in~\cite[p.1114]{St87L}.

\section{$p$-adic continuity of Stienstra's matrices}\label{sec:st_mat}

Let $R$ be a commutative ring and $\Lambda \in R[x_1^{\pm 1},\ldots,x_N^{\pm 1}]$ be a Laurent polynomial in $N$ variables, which we will write as
\[
\Lambda (x) \= \sum_u \; a_u \, x^u\,, \quad a_u \in R\,,
\]
where the summation runs over a finite set of vectors $u=(u_1,\ldots,u_N) \in \Z^N$ and $x^u$ means $x_1^{u_1}\ldots x_N^{u_N}$. The \emph{support} of $\Lambda$ is the set of exponents of monomials in $\Lambda$, which we denote by $\supp(\Lambda) \= \{ u : a_u \ne 0\}$. The \emph{Newton polytope} $\Newt(\Lambda) \subset \R^N$ is the convex hull of $\supp(\Lambda)$. 

Consider the set of internal integral points $J \= \Newt(\Lambda)^o \,\cap\, \Z^N$, where $\Newt(\Lambda)^o$ denotes the topological interior of the Newton poltope. If $\Newt(\Lambda)$ belongs to a hyperplane, we restrict to that hyperplane to define $\Newt(\Lambda)^o$.  Let $h=\#J$ be the number of internal integral points in the Newton polytope. Consider the following sequence of $h \times h$ matrices $\{ \beta_n; n \ge 0\}$ with entries in $R$ whose rows and columns are indexed by the elements of $J$:
\be{beta_mat}
( \beta_n )_{u,v \in J} \= \text{ the coefficient of } x^{(n+1)v - u} \text{ in } \Lambda(x)^n\,.
\ee

Let $p$ be a prime number. Assume that $R$ is endowed with a $p$th-power Frobenius endomorphism, that is a ring endomorphism $\Phi: R \to R$ satisfying 
\be{Frob}\Phi(r) \cong r^p \mod p\,, \quad r \in R\,.
\ee
Below we apply $\Phi$ to $h \times h$ matrices with entries in $R$ entry-wise. This is an endomorphism of the ring of matrices but not a $p$th-power Frobenius: property like~\eqref{Frob} won't be satisfied in general for matrices of size $h>1$. We restrict our attention to the sub-sequence $\alpha_s := \beta_{p^s-1}$, $s \ge 0$. The entries of these matrices are then given by
\[
( \alpha_s )_{u,v \in J} \= \text{ the coefficient of } x^{p^s v - u} \text{ in } \Lambda(x)^{p^s-1}\,.
\]

\begin{theorem}\label{alpha_congs}
\begin{itemize}
\item[(i)] For every $s \ge 1$
\[
\alpha_s \cong  \alpha_1 \cdot \Phi(\alpha_1) \cdot \ldots \cdot \Phi^{s-1}(\alpha_1)    \mod p \,.
\]

\item[(ii)] Assume that all matrices $\alpha_s$ are invertible. Then for every $s \ge 1$
\[
\alpha_{s+1} \cdot \Phi(\alpha_{s})^{-1} \cong \alpha_{s} \cdot \Phi(\alpha_{s-1})^{-1} \mod p^s \,.
\]
\end{itemize}
\end{theorem}

\bigskip

The proof of this theorem will be given in the next two sections. It is unessential that we multiply by the inverse matrices on the right in~(ii). The same congruence is true if one multiplies on the left, which is clear from the proof given in Section~\ref{sec:proof}:
\[
\Phi(\alpha_{s})^{-1} \cdot \alpha_{s+1} \cong \Phi(\alpha_{s-1})^{-1} \cdot \alpha_{s}  \mod p^s \,.
\]
Therefore we have two $p$-adic limits 
\be{lim_mat}
\alpha \= \underset{s \to \infty} \lim \alpha_{s} \cdot \Phi(\alpha_{s-1})^{-1}\,, \qquad \alpha' \= \underset{s \to \infty} \lim \Phi(\alpha_{s-1})^{-1} \cdot \alpha_{s}\,,
\ee
but characteristic polynomials of the two $h \times h$ matrices $\alpha$ and $\alpha'$ are clearly the same.

It seems that matrices \eqref{beta_mat} first appeared in~\cite{St87,St87L} in the relation to the hypersurface $X$ given by the equation $\Lambda(x)=0$. Namely, it is shown in~ \cite[Theorem~1]{St87} that (under the conditions that $R$ is Noetherian and $X$ is a projective hypersurface of degree $d>N$, flat over $R$) the Artin-Mazur functor $H^{N-1}(X,G^{\;\widehat{}}_{m,X})$ is a formal group over $R$ of dimension $h=\binom{d-1}{N}$. If in addition $R$ is flat over $\Z$, the logarithm of this formal group law is given by
\[
\ell(\tau) \= \sum_{m=1}^{\infty} \frac{\beta_{m-1}}{m} \tau^m\,.
\] 
More precisely, the logarithm is the $h$-tuple of power series $\{\ell_u(\tau) ; u \in J\}$ in $h$-tuple of variables $\{\tau_u; u \in J\}$ given by $\ell_u(\tau) \=  \sum_{m=1}^{\infty} \frac1{m} \; \sum_{v \in J} \bigl(\beta_{m-1}\bigr)_{u,v} \tau_v^m$. In~\cite{St87L} another congruences for matrices $\{\beta_n; n \ge 0\}$  were proved in the case when $X$ is a double covering of $\P^n$ and $R$ is flat over $\Z$. These latter congruences generalize those of Atkin and Swinnerton-Dyer for elliptic curves. We discuss them in Section~\ref{sec:ASD} and derive the consequences for the limiting matrices~\eqref{lim_mat}, which will explain the title of this paper. In~\cite[p.1115]{St87L} Jan Stienstra mentions that congruences of Atkin and Swinnerton-Dyer type can be proved for complete intersections using a similar method. This would mean that in a rather general setting eigenvalues of the limiting matrices~\eqref{lim_mat} are eigenvalues of the Frobenius operator acting on the middle crystalline cohomology of the fibre of $X$ at $p$.

Matrix $\alpha_1 \= \beta_{p-1}$ appeared earlier in~\cite{Mill72}: for a smooth projective hypersurface $X$, $\alpha_1$ modulo $p$ is the matrix of the Cartier operator in a certain basis $\{ w_u ; u \in J \}$ of $H^0(\overline{X}; \Omega^{N-1}_{\overline{X}/\F_p})$. The details are given in Section~\ref{sec:Cartier}. 

It seems that congruences similar to those of Theorem~\ref{alpha_congs} were considered in~\cite{Katz85}. However, our statement is explicit and for the moment we will not investigate this possible relation.     

\section{Main lemma on congruences for the powers of $\Lambda(x)$}

The ideas in this section are due to Anton Mellit and were introduced in \cite{MV13}. We reproduce them here for the sake of completeness, and also because we need a formulation with Frobenius. 

We extend the Frobenius morphism $\Phi: R \to R$ to the ring of Laurent polynomials $R[x_1^{\pm 1},\ldots,x_N^{\pm 1}]$ by setting $\Phi(x_i)=x_i^p$, so that
\[
\Phi \Bigl( \sum_u \; a_u \, x^u \Bigr) \=  \sum_u \; \Phi(a_u) \, x^{pu}\,. 
\]
It is again a $p$th-power Frobenius morphism of the ring of Laurent polynomials: $\Phi(\Gamma) \cong \Gamma^p \mod p$ for any Laurent polynomial $\Gamma(x)$.

For a non-negative integer $n$, let $\ell(n)$ be the length of the expansion of $n$ to the base $p$:
\[
n \= n^{(0)} \+ n^{(1)} \, p  \+ n^{(2)} \, p^2 \+ \ldots \+ n^{(\ell(n)-1)}\, p^{\ell(n)-1}\,,\quad 0 \le n^{(i)} \le p-1\,.
\]
For two integers $n_1, n_2 \ge 0$ we denote by
\[
n_1 * n_2 := n_1 + p^{\ell(n_1)} n_2 
\] 
the integer whose expansion to the base $p$ is the concatenation of the respective expansions of $n_1$ and $n_2$.

With a Laurent polynomial $\Lambda(x)$ we associate a sequence of Laurent polynomials $\{ I_{\Lambda,n}(x); n \ge 0 \}$ uniquely defined by the condition that for every $n \ge 0$
\be{Is}
\Lambda^n \= \sum_{n_1* \ldots *n_r=n} I_{\Lambda,n_1} \cdot \Phi^{\ell(n_1)} \Bigl( I_{\Lambda,n_2} \Bigr) \cdot \Phi^{\ell(n_1)+\ell(n_2)}\Bigl( I_{\Lambda,n_3} \Bigr) \cdot \ldots \cdot \Phi^{\sum_{j=1}^{r-1}\ell(n_j)}\Bigl( I_{\Lambda,n_r} \Bigr) \,.
\ee
The summation here runs over all possible tuples of non-negative integers whose expansions to the base p concatenate to the expansion of $n$, so the number of parts $r$ varies between $1$ and $\ell(n)$. For every $n$ the right-hand side consists of the term $I_{\L,n}$ plus the sum of products of various $I_{\L,n'}$ with $n'<n$ and their images under powers of Frobenius. Therefore this formula defines all $I_{\L,n}$ recursively.

\begin{lemma}\label{A_Lemma} Polynomials $\{ I_{\Lambda,n}; n \ge 0\}$ enjoy the following properties:
\begin{itemize}
\item[(i)] $\Newt(I_{\Lambda,n}) \subseteq n \; \Newt(\Lambda)$;
\item[(ii)] $I_{\Lambda,n} \cong 0 \mod p^{\ell(n)-1}$.
\end{itemize}
\end{lemma}

\begin{proof} Part (i) easily follows by induction. To prove (ii) an alternative construction of the polynomials $I_{\Lambda,n}$ can be given, in which the congruence becomes apparent. 

For a Laurent polynomial $\Gamma(x)$ define a sequence of Laurent polynomials $\{R_s(\Gamma) ; s \ge 0\}$ by  
\[
R_s(\Gamma) \;:=\; \Gamma^{p^s} - \Phi(\Gamma)^{p^{s-1}}
\]
when $s \ge 1$ and $R_0(\Gamma):=\Gamma$. There polynomials were called \emph{ghost terms} of $\Gamma$ in~\cite{MV13}. The following properties of ghost terms easily follow by induction on $s \geq 0$:
\begin{itemize}
\item $\Lambda(x)^{p^s} \= R_0 (\Phi^s(\Lambda)) + R_1(\Phi^{s-1} (\Lambda)) + \cdots + R_{s-1}(\Phi(\Lambda)) +  R_s(\Lambda)$;
\item $ R_s(\Lambda) \cong 0 \mod \; p^s$;
\item $\Newt(R_s(\L)) \subseteq p^s \, \Newt(\L)$.
\end{itemize}

Expanding an integer $n$ to the base $p$ as $n \= n^{(0)} \+ n^{(1)} \, p  \+ n^{(2)} \, p^2 \+ \ldots \+ n^{(\ell(n)-1)}\, p^{\ell(n)-1}$ with digits $0 \leq n^{(i)} \leq p-1$ we use the first property of ghost terms to decompose the product
\[
\Lambda^{n} = \Lambda^{n_0} (\Lambda^{n_1})^p \ldots (\Lambda^{n_{\ell(n)-1}})^{p^{\ell(n)-1}}
\]
as the sum of products of ghost terms of images under powers of $\Phi$ of the collection of $p$ Laurent polynomials $\Lambda^{a}$, $0 \le a \le p-1$. We obtain that 
\[
\Lambda^{n} \= \sum_{\begin{small}\bal m=(&m_0, m_1,\ldots,m_{\ell(n)-1}) \\
& 0 \le m_i \le i \eal\end{small}} R_m(n,\L)
\] 
where
\[
R_m(n,\L) := \prod_{i=0}^{\ell(n)-1} R_{m_i}(\Phi^{i-m_i}(\L^{n^{(i)}})) 
\]
and the summation above runs over the set of all integral tuples $m =(m_0,m_1,\ldots,m_{\ell(n)-1})$  of length $\ell(n)$ satisfying $0 \leq m_i \le i$. For such a tuple we denote $|m|=\sum_i m_i$. The second and third properties of ghost terms yield respectively
\begin{itemize}
\item $R_m(n,\Lambda) \cong 0 \mod\; p^{|m|}$;
\item $\Newt\Bigl(R_m(n,\Lambda) \Bigr) \subseteq n \; \Newt(\Lambda)$.
\end{itemize}
We consider the operation of concatenation on tuples
\[
(m_0,\dots,m_{k-1}) * (m_0',\dots,m_{s-1}') := (m_0,\dots,m_{k-1},m_0',\dots,m_{s-1}')\,.
\]
A tuple $m$ satisfying $0 \le m_i \le i$ is called \emph{indecomposable} if it can not be written as a concatenation of two shorter tuples with the same property. If $m$ is an indecomposable tuple of length $k$ then $|m|\ge k-1$, which can be proved by induction in $k$. For each $n \ge 0$ consider the Laurent polynomial
\[
J_{\Lambda,n} := \sum_{\begin{small}\bal & m \text{ indecomposable}\\ &\text{ of length } \ell(n) \eal\end{small}} R_m(n,\L)\,.
\]
It is clear that $J_{\Lambda,n} \cong 0 \mod p^{\ell(n)-1}$ because $R_m(n,\Lambda)  \cong 0 \mod p^{\ell(n)-1}$ for every indecomposable tuple $m$ of length $\ell(n)$. 

Every tuple $m$ of length $\ell(n)$ can be uniquely written as a concatenation of indecomposable tuples $m \= m_1 * \ldots * m_r$ for some $1 \le r \le \ell(n)$, and if we chop the expansion of $n$ to the base $p$ into $r$ blocks of respective lengths, $n=n_1*\ldots *n_r$, we then have
\[
R_m(n,\L) \= R_{m_1}(n_1,\L) \cdot \Phi^{\ell(n_1)}\Bigl( R_{m_2}(n_2,\L)\Bigr) \cdot \Phi^{\ell(n_1)+\ell(n_2)}\Bigl( R_{m_3}(n_3,\L)\Bigr) \cdot \ldots \cdot \Phi^{\sum_{j=1}^{r-1}\ell(n_j)}\Bigl( R_{m_r}(n_r,\L)\Bigr) \,.
\]
Here $n_i$'s are blocks of $p$-adic digits rather then integers, but the notation clearly extends to this case. Suppose $n_i$ is a block of digits of length $\ell(n_i)>1$ with the right-most digit being a zero. Since the tuple $m_i$ is indecomposable, its right-most number is greater than $0$. This implies $R_{m_i}(n_i,\L)=0$ because all higher ghost terms of the constant polynomial $\L^0=1$ vanish. It follows that
\[\bal
\Lambda^n &\= \sum_{\begin{small}\bal & \text{ all tuples } m\\ & \text{ of length } \ell(n) \eal\end{small}} R_m(n,\L)\,\\
&\= \sum_{n_1* \ldots *n_r=n} J_{\Lambda,n_1} \cdot \Phi^{\ell(n_1)} \Bigl( J_{\Lambda,n_2} \Bigr) \cdot \Phi^{\ell(n_1)+\ell(n_2)}\Bigl( J_{\Lambda,n_3} \Bigr) \cdot \ldots \cdot \Phi^{\sum_{j=1}^{r-1}\ell(n_j)}\Bigl( J_{\Lambda,n_r} \Bigr) \,,
\eal\]
where in the upper sum tuples $m=(m_0,\ldots,m_{\ell(n)-1})$ are assumed to satisfy the usual condition $0 \le m_i \le i$ for all $0\le i \le \ell(n)-1$ and the lower summation runs over all possible tuples of non-negative integers whose expansions to the base p concatenate to the expansion of $n$. We see that the sequence of polynomials $\{ J_{\L,n}; n \ge 0 \}$ satisfies the defining relation~\eqref{Is}  for the sequence $\{ I_{\L,n}; n \ge 0 \}$. As we remarked earlier, such a sequence is unique. Therefore $J_{\L,n}=I_{\L,n}$ for every $n$, which finishes the proof of the lemma. 
\end{proof}

\section{Proof of Theorem 1}\label{sec:proof}

Define a new sequence of $h \times h$ matrices $\{ \gamma_s; s \ge 0\}$ so that for every $s$
\be{transform}
\alpha_s \= \sum_{s_1+\ldots+s_r=s} \gamma_{s_1} \cdot \Phi^{s_1}(\gamma_{s_2}) \cdot \Phi^{s_1+s_2}(\gamma_{s_3}) \cdot \ldots \cdot \Phi^{\sum_{j=1}^{r-1} s_j}(\gamma_{s_r}) \,,
\ee
where the summation runs over all ordered partitions of $s$ into a sum of non-negative integers:
\[
\sum_{s_1+\ldots+s_r=s} \= \qquad \sum_{r=1}^{s} \quad \sum_{\begin{tiny}\bal &(s_1,\dots,s_r) \in \Z_{\ge 1}^r \\ & s_1+\ldots+s_r=s  \eal\end{tiny}} \quad .
\] 
Formula~\eqref{transform} defines $\gamma_s$ recursively, since in the right-hand side we have $\gamma_s$ plus the sum of products of various Frobenius images of matrices $\gamma_{s'}$ with $s' < s$.

\begin{lemma}\label{trans_cong}
$\gamma_s \cong 0 \mod p^{s-1}$ for every $s \ge 1$.
\end{lemma}
\begin{proof} During this proof, let's forget that matrices $\gamma_s$ were already defined and consider \[
(\gamma_s)_{u,v} := \Bigl[ I_{\L,p^s-1}\Bigr]_{x^{p^s v - u}} \,.
\]
If we show that they satisfy~\eqref{transform}, the statement will follow from (ii) in Lemma~\ref{A_Lemma}. We have
\be{trans_1}
\Lambda^{p^s-1} \= \sum_{s_1+\ldots+s_r=s} I_{\L,p^{s_1}-1} \cdot \Phi^{s_1} \Bigl( I_{\L,p^{s_2}-1} \Bigr) \cdot \Phi^{s_1+s_2}\Bigl( I_{\L,p^{s_3}-1} \Bigr) \cdot \ldots \cdot \Phi^{\sum_{j=1}^{r-1}s_j}\Bigl( I_{\L,p^{s_r}-1} \Bigr)\,.
\ee
Let's consider one term in the sum corresponding to $(s_1,\dots,s_r)$. For $k= 0, 1, \ldots, r-1$ let's denote 
\[
A_k=I_{\L,p^{s_{k+1}}-1} \cdot \Phi^{s_{k+1}} \Bigl( I_{\L,p^{s_{k+2}}-1} \Bigr) \cdot \ldots \cdot \Phi^{\sum_{j=k+1}^{r-1}s_j}\Bigl( I_{\L,p^{s_r}-1} \Bigr)\,,
\]
so that the term under consideration is $A_0$ and for each $k$ we have $A_k=I_{\Lambda,p^{s_{k+1}}-1}\cdot \Phi^{s_{k+1}}(A_{k+1})$. Let $u,v \in J$. To compute the coefficient near $x^{p^s v - u}$ in $A_0=I_{\Lambda,p^{s_1}-1}\cdot \Phi^{s_1}(A_1)$, we notice that $\supp(I_{\Lambda,p^{s_1}-1}) \subseteq (p^{s_1}-1) \Newt(\Lambda)$ by (i) in Lemma~\ref{A_Lemma} and $\supp(\Phi^{s_1}(A_1)) \subseteq p^{s_1}  \supp(A_1) \subset p^{s_1} \Z^N$. Suppose $w \in \supp(I_{\Lambda,p^{s_1}-1})$ and $\tau \in \supp(A_1)$ are such that
\be{int_pts}
w + p^{s_1} \tau \= p^{s} v \- u\,.
\ee
Since $w \in (p^{s_1}-1) \,\Newt(\Lambda)$ and $u \in \Newt(\L)^o$, we have that $w+u \in p^{s_1} \,\Newt(\L)^o$. Moreover,~\eqref{int_pts} implies that $w + u \in p^{s_1}\Z^N$, and therefore
$\frac{1}{p^{s_1}} (w + u) \in \Newt(\L)^o \cap \Z^N \= J$. On the other hand, for every $\mu \in J$ vectors
\[
w \= p^{s_1} \mu \- u \,,\qquad \tau \= p^{s-s_1} v \- \mu  
\]
satisfy~\eqref{int_pts}. It follows that
\[
\Bigl[ I_{\Lambda,p^{s_1}-1}\cdot \Phi^{s_1}(A_1) \Bigr]_{x^{p^s v - u}} \= \sum_{\mu \in J} \bigl[ I_{\Lambda,p^{s_1}-1} \bigr]_{x^{p^{s_1} \mu - u}} \cdot \Phi^{s_1}\Bigl( \bigl[ A_1 \bigr]_{p^{s-s_1}v-\mu} \Bigr)\,,  
\]
where we recognize matrix multiplication in the right-hand side. We then proceed by splitting $A_1=I_{\L,p^{s_2}-1} \cdot \Phi^{s_2}(A_2)$, and so on. Finally we obtain that 
\[\bal
\Bigl[ I_{\L,p^{s_1}-1} \cdot \Phi^{s_1} \Bigl( I_{\L,p^{s_2}-1} \Bigr) \cdot \Phi^{s_1+s_2}\Bigl( I_{\L,p^{s_3}-1} \Bigr) \cdot \ldots \cdot \Phi^{\sum_{j=1}^{r-1}s_j}\Bigl( I_{\L,p^{s_r}-1} \Bigr)\Bigl]_{x^{p^s v - u}}\\
\= \Bigl( \gamma_{s_1} \cdot \Phi^{s_1}(\gamma_{s_2}) \cdot \Phi^{s_1+s_2}(\gamma_{s_3}) \cdot \ldots \cdot \Phi^{\sum_{j=1}^{r-1}s_j}(\gamma_{s_r}) \Bigr)_{u,v} \,.
\eal\]
Summation over all partitions $s_1+\ldots+s_r=s$ in~\eqref{trans_1} gives us~\eqref{transform}.
\end{proof}

\begin{proof}[Proof of Theorem~\ref{alpha_congs}] In~\eqref{transform} terms corresponding to partitions $s_1+\dots+s_r=s$ containing at least one $s_i>1$ vanish modulo $p$ due to Lemma~\ref{trans_cong}. The only partition without $s_i > 1$ is $1+\ldots+1=s$, and therefore
\[
\alpha_s \cong \gamma_1 \cdot \Phi(\gamma_1) \cdot \ldots \cdot \Phi^{s-1}(\gamma_1) \mod p \,. 
\] 
Part (i) follows immediately since $\alpha_1 = \gamma_1$.

We will prove (ii) by induction on $s$. The case $s=1$ follows from part (i). From~\eqref{transform} we have
\[
\alpha_s \= \gamma_1 \, \Phi(\alpha_{s-1}) \+ \gamma_2 \, \Phi^2(\alpha_{s-2}) \+ \ldots \+ \gamma_{s-1} \Phi^{s-1}(\alpha_1) \+ \gamma_s
\]
for every $s \ge 1$. Substituting these expressions into the two sides of the congruence that we want to prove we get
\[\bal
&\alpha_{s+1} \, \Phi(\alpha_s)^{-1} = \gamma_1 + \sum_{j=2}^{s+1} \gamma_j \Phi^j(\alpha_{s+1-j}) \, \Phi(\alpha_s)^{-1} \,,\\
&\alpha_{s} \, \Phi(\alpha_{s-1})^{-1} = \gamma_1 + \sum_{j=2}^{s} \gamma_j \Phi^j(\alpha_{s-j}) \, \Phi(\alpha_{s-1})^{-1}\,. \\
\eal\]
Since we want to compare these two expressions modulo $p^s$ and $\gamma_{s+1} \cong 0 \mod p^s$, the last term in the upper sum can be ignored. For every $j=2,\ldots,s$ we use inductional assumption as follows:
\[\bal
& \alpha_s \Phi(\alpha_{s-1})^{-1} \cong \alpha_{s-1} \Phi(\alpha_{s-2})^{-1} \mod p^{s-1}\\
& \alpha_{s-1} \Phi(\alpha_{s-2})^{-1} \cong \alpha_{s-2} \Phi(\alpha_{s-3})^{-1} \mod p^{s-2}\\
& \vdots \\
& \alpha_{s+2-j} \Phi(\alpha_{s+1-j})^{-1} \cong \alpha_{s+1-j} \Phi(\alpha_{s-j})^{-1} \mod p^{s+1-j}\\
\eal\]
We then apply the respective power of $\Phi$ to each row and multiply these congruence out to get that modulo $p^{s+1-j}$
\[\bal
\alpha_s \Phi^{j-1}(\alpha_{s+1-j})^{-1} & \= \alpha_s \Phi(\alpha_{s-1})^{-1} \Phi(\alpha_{s-1}) \Phi^2(\alpha_{s-2})^{-1} \ldots \Phi^{j-1}(\alpha_{s+1-j})^{-1}\\
& \cong \alpha_{s-1} \Phi(\alpha_{s-2})^{-1} \Phi(\alpha_{s-2}) \Phi^2(\alpha_{s-3})^{-1} \ldots \Phi^{j-1}(\alpha_{s-j})^{-1} \= \alpha_{s-1} \Phi^{j-1}(\alpha_{s-j})^{-1}\,.
\eal\]
Since determinants of these matrices are non-zero modulo $p$, we can invert them to get
\[
\Phi^{j-1}(\alpha_{s+1-j}) \alpha_s^{-1} \cong \Phi^{j-1}(\alpha_{s-j})  \alpha_{s-1}^{-1} \mod p^{s+1-j}\,. 
\]
Now we allpy $\Phi$ and multiply by $\gamma_j$. Since $\gamma_j \cong 0 \mod p^{j-1}$ we get
\[
\gamma_j \Phi^j(\alpha_{s+1-j}) \Phi(\alpha_s)^{-1} \cong \gamma_j\Phi^{j}(\alpha_{s-j})  \Phi(\alpha_{s-1})^{-1} \mod p^{s}\,. 
\]
Summation in $j$ gives the desired result.
\end{proof}

\section{The case of a smooth projective hypersurface}\label{sec:Cartier}

Let's assume for simplicity that $R = \Z$ and $\Phi$ is the identity transformation. Let $F \in \Z[X_0,\ldots,X_N]$ be an absolutely irreducible homogeneous polynomial of degree $d > N$, such that the hypersurface 
\be{hyp}
X \;:\; F(X_0,\ldots,X_N)\=0 
\ee
is smooth. In this case 
\[
\dim H^0(X,\Omega_{X/R}^{N-1}) \= \binom{d-1}{N}
\]
(see e.g.~\cite{Nic} for this computation) is equal to the number of internal integral points in the Newton polytope $\Newt(F) \subset \R^{N+1}$. Indeed, for a generic polynomial $F$ (that is, all monomials of degree $d$ are present with nonzero coefficients) $\Newt(F)$ is the simplex in the hyperplane $\sum_{i=0}^{N}{u_i}=d$ with $N+1$ vertices of the form $(0,\ldots,0,d,0,\ldots,0)$, and the set of integral internal points is given by 
\be{int_pt}
J \= \{ u = (u_0,\ldots,u_N) \in \Z^{N+1} \;:\; u_i \ge 1 \quad \forall i\,,\;\sum_{i=0}^N u_i = d\}\,.
\ee
When $X$ is smooth, one can show that we still have $\Newt(F)^\circ \cap \Z^{N+1} = J$. 

An explicit basis in $H^0(X,\Omega_{X/R}^{N-1})$ corresponding to the elements of the set $J$ in~\eqref{int_pt} can be given as follows. Let $f(x_1,\dots,x_N)=F(1,x_1,\ldots,x_N)$. W.l.o.g. we can assume that $\frac{\partial f}{\partial x_{N}} \not \equiv 0$ on $X$. Then the forms
\be{J_basis}\bal
\omega_u \= & x_1^{u_1-1} \cdot \ldots \cdot x_N^{u_N-1} \, \frac{dx_1 \wedge \ldots \wedge x_{N-1}}{\frac{\partial f}{\partial x_N}(x_1,\ldots,x_N)} \\
& u = (u_0,\ldots,u_N) \in J
\eal\ee
consitute a basis for $H^0(X,\Omega_{X/R}^{N-1})$.  

Let $p$ be a prime and $\overline{F} \in \F_p[X_0,\dots,X_N]$ be the reduction of $F$ modulo $p$, and suppose that
\[
\overline{X} \;:\; \overline F(X_0,\ldots,X_N)\=0 
\] 
is again a smooth hypersurface over $\F_p$. It is claimed in~\cite[Corollary~1]{Mill72} that the matrix of the Cartier operator
\[
C : H^0(\overline{X},\Omega_{\overline{X}/\F_p}^{N-1}) \; \to \; H^0(\overline{X},\Omega_{\overline{X}/\F_p}^{N-1})
\]
in the basis corresponding to~\eqref{J_basis} is given by
\[\bal
(C)_{u,v \in J} &\= \text{ the coefficient of } X^{p v - u} \text{ in } \overline F(X)^{p-1} \\
&\= \alpha_1 \mod p\,,
\eal\] 
where $\alpha_1\=\beta_{p-1}$ in the notation of Section~\ref{sec:st_mat} with $\Lambda=F$.

We see that $\alpha_1$ is invertible as a matrix with $\Z_p$-entries if and only if the Cartier operator is invertible. If $\alpha_1$ is invertible, then $\det(\alpha_1) \not \cong 0 \mod p$ and $\det(\alpha_s) \cong  \det(\alpha_1)^s  \not \cong 0   \mod p$, so all $\alpha_s$ are invertible as well.

\section{Application of Atkin and Swinnerton-Dyer type congruences}\label{sec:ASD}

In this section we combine Theorem~\ref{alpha_congs} with the main result of~\cite{St87L}. Here again $R=\Z$ and $\Phi$ is trivial. Let $G(X_0,\ldots,X_N)$ be a polynomial with integral coefficients, homogeneous of degree $2d$, and $d>N$. Let $X$ be the double covering of $\P^{N}$ defined by the equation $W^2 \= G(X_0,\ldots,X_N)$. Let 
\[
\L(W,X_0,\ldots,X_N) \= W^2 \- G(X_0,\ldots,X_N)\,.
\]

Let $J$ be the set from~\eqref{int_pt}, that is
\[J \= \Bigl\{ u=(u_0,\ldots,u_N) \in \Z^{N+1} : u_i \ge 1 \; \forall i \;,\;\sum_{i=0}^N u_i = d\Bigr\}\,.\]
For a generic $G$ (that is, every monomial of degree $2d$ is present with a nonzero coefficient), we claim that $\Newt(\L)^{\circ}\cap \Z^{N+2} \= \{ (1,u) : u \in J\}$.
Indeed, $\Newt(G)$ is the simplex in the hyperplane $\sum_{i=0}^{N}{u_i}=2d$ with $N+1$ vertices of the form $(0,\ldots,0,2d,0,\ldots,0)$, and 
\[
\Newt(G)^0 \= \{ u \in \R^{N+1} \;:\; u_i > 0 \; \forall i \;,\;\sum_{i=0}^N u_i = 2 d  \}\,. 
\]
$\Newt(\Lambda)$ is the convex hull of $(2,0,\ldots,0)$ and $0 \times \Newt(G)$ in $\R^{N+2}$, so our claim follows.

When $X$ is smooth, it is shown in~\cite[Appendix A4]{St87} that $\dim H^0(X,\Omega_{X/R}^N) \= \binom{d-1}{N}$, and an explicit basis of global $N$-forms indexed by the elements of $J$ is constructed.

\bigskip

Let $h=\# J = \binom{d-1}{N}$. We will assume further on that $\Newt(\L)^{\circ}\cap \Z^{N+2} \= \{ (1,u) : u \in J\}$. We consider $h \times h$ matrices from Section~\ref{sec:st_mat}:
\[
\bigl(\beta_n\bigr)_{u,v \in J} \= \text{ the coefficient of } W^{n}X^{(n+1)v - u} \text{ in } \Lambda^{n} \= \bcs 0\,, \quad n \text{ odd},  \\
(-1)^{\frac{n}2} \binom{n}{n/2}  \cdot \delta_n \,, \quad n \text{ even},\\
\ecs\\
\]
where
\[
\bigl(\delta_n\bigr)_{u,v \in J} \= \text{ the coefficient of } X^{(n+1)v - u} \text{ in } G^{\frac{n}2}\,.
\]

\bigskip

Let's now fix a prime $p$ and consider the subsequence of matrices $\alpha_s \= \beta_{p^s-1}$, $s \ge 0$. Assuming that $\det(\alpha_1) \ne 0 \mod p$, we have that all $\alpha_s$ are invertible over $\Z_p$ by (i) of Theorem~\ref{alpha_congs} and by~(ii) there exists the limiting $h \times h$ matrix
\[
\alpha \;:= \;  \underset{s \to \infty} \lim \alpha_s \alpha_{s-1}^{-1}
\]
with entries in $\Z_p$. 

Let $\overline{X} \= X \times_{Spec(\Z)} Spec(\F_p)$ be the fibre of $X$ at $p$, and suppose that there exists a smooth projective variety $\overline{Y}$ over $\F_p$ and a morphism $\pi: \overline{Y} \to \overline{X}$ such that $\pi_* O_{\overline{Y}} = O_{\overline{X}}$ and $R^i \pi_* O_{\overline{Y}}  = 0$ for $i \ge 1$. Consider the (reciprocal) characteristic polynomial of the $p$th power Frobenius $\it{Frob}_p$ acting on the middle crystalline cohomology of $\overline{Y}$
\[
\det \Bigl( 1 - T \cdot \it{Frob}_p | H^N_{cris}(\overline{Y}) \otimes \Q \Bigr) \= 1 \+ a_1 T \+ \ldots \+ a_k T^k \;\in\; \Z[T]\,.
\]
Let $n$ be a positive integer and $\nu$ be the maximal integer such that $p^{\nu} |n$. By~\cite[Theorem 0.1]{St87L}, if $\nu \ge h$ then
\be{ASD}
\delta_{n-1} \+ a_1 \,\delta_{\frac {n}q - 1 } \+ \ldots \+ a_k \,\delta_{\frac{n}{q^k}-1} \equiv 0 \mod p^{\nu - h+1}\,,
\ee
where we assume that $\delta_n = 0$ for $n \not \in 2\mathbb{N}$.

\bigskip

\begin{corollary} When $\det(\alpha_1) \ne 0 \mod p$, the matrix $\alpha = \underset{s \to \infty} \lim \alpha_s \alpha_{s-1}^{-1}$
satisfies
\[
\alpha^k \+ a_1 \alpha^{k-1} \+ \ldots \+ a_{k-1} \alpha \+ a_k \= 0\,. 
\]
\end{corollary}
\begin{proof}
The sequence 
\[
b_n \= \bcs 0\,, \quad n \text{ odd},  \\
(-1)^{\frac{n}2} \binom{n}{n/2}   \,, \quad n \text{ even}\\
\ecs
\]
is the sequence of constant terms of the polynomial $\L(x)=x-\frac1{x}$, to which Theorem~\ref{alpha_congs} is applicable. It follows that 
\[
\frac{b_{p^{s+1}-1}}{b_{p^s-1}} \equiv \frac{b_{p^s-1}}{b_{p^{s-1}-1}} \mod p^s
\]
for every $s \ge 1$. One can show, using for instance $p$-adic gamma function, that actually
\[
\frac{b_{p^s-1}}{b_{p^{s-1}-1}} \equiv 1 \mod p^s\,,
\]
and therefore
\[
\alpha_s \cdot \alpha_{s-1}^{-1} \equiv \delta_{p^s-1} \cdot \delta_{p^{s-1}-1}^{-1} \mod p^s  
\]
and
\[
\alpha \= \underset{s \to \infty} \lim \delta_{p^s-1} \, \delta_{p^{s-1}-1}^{-1}\,.
\]
The result now follows from Stienstra's congruences~\eqref{ASD}.
\end{proof}

\bigskip

We see that eigenvalues of $\alpha$ are also eigenvalues of the $p$th power Frobenius $\it{Frob}_p$. Since $\det(\alpha) \cong \det(\alpha_1)\mod p$ and  $\det(\alpha_1) \ne 0 \mod p$, the eigenvalues of $\alpha$ are $p$-adic units.

\section{An example}

Consider the hyperelliptic curve
\[
y^2 \= x^5 \+ 2 x^2 \+ x \+ 1\,.
\]
Stienstra's matrices here will have size $2 \times 2$:
\[
\beta_n \= \text{ the coefficients of } \begin{pmatrix} x^n y^n & x^{2 n+1} y^n \\ x^{n-1} y^n & x^{2n} y^n\end{pmatrix} \text{ in } (y^2 - x^5 - 2 x^2 - x - 1)^n \,.
\]
For example, with $p=11$ we have 

\bigskip
\begin{tabular}{c|cccc}
s & 0 & 1 & 2 & 3 \\
\hline
$\alpha_s = \beta_{p^s-1}$ & $\begin{pmatrix}1&0\\0&1\end{pmatrix}$ & $\begin{pmatrix}-81144&-1260\\-81900&-1260\end{pmatrix}$ & $\ldots$ & $\ldots$ \\
${\rm tr}(\alpha_s \cdot \alpha_{s-1}^{-1}) \mod p^s$ & & $8 + O(11)$& $8 + 11 + O(11^2)$ & $8 + 11 + 11^2 + O(11^3)$ \\
${\rm det}(\alpha_s \cdot \alpha_{s-1}^{-1}) \mod p^s$ & & $7 + O(11)$& $7 + 6 \cdot 11 + O(11^2)$ & $7 + 6 \cdot 11 + 3 \cdot 11^2 + O(11^3)$ \\
\end{tabular}

\bigskip

Using Kedlaya's algorithm we computed the reciprocal characteristic polynomial of the Frobenius on the first crystalline cohomology of the curve:
\[\bal
\det \Bigl(1 - T \cdot \it{Frob}_{11} \;|\; H^1_{cris}(C) \Bigr) & \= 1 \+ 3 T \+ 18 T^2 \+ 3 \cdot 11 T^3 \+ 11^2 T^4  \\
& \= (1 \+ 4 T \+ 11 T^2) (1 \- T \+ 11 T^2)
\eal\]
The eigenvalues of the Frobenius are
\[
\lambda_{1,2} \= -2 \pm \sqrt{-7}\,, \qquad \lambda_{3,4} \= \frac{1 \pm \sqrt{-43}}2 \,.
\]
Both $-7$ and $-43$ are squares modulo $11$, and $11$-adic unit eigenvalues are
\[\bal
\lambda_1 &\= 7 \+ 2 \cdot 11 \+ 2 \cdot 11^2 \+ O(11^3)\\
\lambda_3 &\= 1 \+ 10 \cdot 11 \+ 9 \cdot 11^2 \+ O(11^3)\\
\eal\]
We see that in the above table traces and determinants converge to
\[\bal
\lambda_1 \+ \lambda_3 &\= 8 + 11 + 11^2 + O(11^3)\\
\lambda_1 \cdot \lambda_3 &\= 7 + 6 \cdot 11 + 3 \cdot 11^2 + O(11^3)\\
\eal\]
respectively.

\bigskip

We hope that on practice our method can be used in the other direction, in order to compute Frobenius data modulo a given power of $p$ using coefficients of powers of the equation of the hypersurface. 

\end{document}